\documentclass[a4paper,fleqn,10pt]{amsart}

\usepackage{amssymb, comment}
\usepackage{mathrsfs}
\usepackage{natbib}
\usepackage{mathscinet}
\usepackage[colorlinks=true]{hyperref}
\hypersetup{
	pdftitle={Polling},
	pdfsubject={Probability Theory},
	pdfauthor={Kamil Kosi\'nski},
	pdfdisplaydoctitle=true,
}

\newcommand{\ee}{\mathbb E}
\newcommand{\pp}{\mathbb P}
\newcommand{\rr}{\mathbb R}

\newcommand{\prob}[1]{\pp\left( #1 \right)}

\newcommand{\bz}{\boldsymbol z}
\newcommand{\bo}{\boldsymbol \omega}
\newcommand{\QL}{\mathscr Q}
\newcommand{\WO}{\mathscr W}
\newcommand{\V}{\mathscr V}
\newcommand{\Ss}{\mathscr S}

\newtheorem{theorem}{Theorem}

\theoremstyle{remark}
\newtheorem{remark}{Remark}

\renewenvironment{proof}[1][\proofname]{\par \normalfont \trivlist
 \item[\hskip\labelsep\itshape #1]\ignorespaces
}{
 \hspace*{\fill}$\Box$ \endtrivlist
}
\renewcommand{\proofname}{\noindent {\bf Proof}}

\begin{document}
\bibliographystyle{plainnat}
\setcitestyle{numbers}

\title{Queue lengths and workloads in polling systems}

\author{O.J.\ Boxma}
\address{E{\sc urandom} and Department of Mathematics and Computer Science,
Eindhoven University of Technology, the Netherlands}
\email{Boxma@win.tue.nl}

\author{O. Kella}
\address{Department of Statistics, The Hebrew University of Jerusalem, Jerusalem 91905, Israel}
\email{Offer.Kella@huji.ac.il}
\thanks{The second author is supported by grant No. 434/09 from the Israel Science Foundation and the Vigevani Chair in Statistics}

\author{K.M.\ Kosi\'nski}
\address{E{\sc urandom}, Eindhoven University of Technology;
Korteweg-de Vries Institute for Mathematics,
University of Amsterdam, the Netherlands}
\email{Kosinski@eurandom.tue.nl}
\thanks{The third author is supported by NWO grant 613.000.701.}

\date{\today}


\begin{abstract}
For a single-server multi-station polling system we focus on the generating function and Laplace-Stieltjes transform of the time-stationary joint queue length and workload distributions, respectively, under no further assumptions on the service discipline. We express these quantities as expressions involving the generating functions of the joint queue length distribution at visit beginnings and visit completions at the various stations. The latter is known for a broad variety of cases.  Finally, we identify a workload decomposition result.
\end{abstract}

\maketitle

\section{Introduction}
\label{sec:INT}
Consider a queueing model consisting of multiple queues: $Q_1,\ldots, Q_N$, $N\ge 1$, attended
by a single server, visiting the queues one at a time in a cyclic
order. Moving from one queue to another, the server incurs a
(possibly zero) switchover time. Such single-server multiple-queue
models are commonly referred to as {\it polling systems}.
Triggered by a wide variety of applications, polling systems have
been extensively studied in the literature; see
\cite{Takagi91a,Takagi00,Wisnia06} for a series of comprehensive
surveys and \cite{Boon10,Levy90,Takagi91b} for extensive overviews of the applicability of polling systems.

In the polling literature much attention has been given to the determination 
of the probability generating function (PGF)
of the joint queue length distribution of the polling system at various epochs. 
In particular, the PGF's of the joint queue length distribution at the epochs 
at which the server begins a visit, $\V_{b_i}(\cdot)$, and completes a visit, 
$\V_{c_i}(\cdot)$, at $Q_i$ were extensively studied. 
A seminal result has been obtained by \citet{Resing93}, 
who found these PGF's for polling systems with service disciplines satisfying the so-called 
{\it branching property}. 
Service disciplines -- rules which establish the server's behaviour while visiting a queue --
that satisfy the branching property include the well known {\it exhaustive} discipline 
(per visit the server continues to serve all customers at a station until it empties) 
and {\it gated} discipline (per visit the server serves only those customers at a station 
which are found there upon its visit).

Polling systems with disciplines which do not satisfy the branching property usually defy exact analysis. A typical example of a discipline that does not satisfy the branching property is the {\it $k$-limited} discipline (per visit at most $k$ customers are served at a station). Nevertheless, it is possible to compute $\V_{b_i}$ and $\V_{c_i}$ for some special cases like completely symmetrical queues or like the two-queue case.  
For example, for the case $N=2$ with $Q_1$ served exhaustively and $Q_2$ according to the $k$-limited discipline,
the formulas for $\V_{b_i}$ and $\V_{c_i}$ were found by \citet{Winands09}.
Furthermore, for the case of $N=2$ and $Q_1$, $Q_2$ both served according to the $1$-limited discipline, the formulas for $\V_{b_i}$ and $\V_{c_i}$ were
found by solving a boundary value problem \cite{Boxma88}. 
For other examples of polling systems with non-branching disciplines that admit a solution for these PGF's we refer to the surveys \cite{Takagi97,Takagi00}.

Although the PGF's $\V_{b_i}$ and $\V_{c_i}$ have been found in a very broad variety of cases, the PGF $\QL(\cdot)$ of the steady-state joint queue length distribution at an {\it arbitrary} epoch, which 
in principle is the most important PGF, received little attention and has not been studied in full generality, one exception (that we are aware of) being \citet{Sidi92}. The first contribution of the present paper is to determine $\QL$ for a quite general class of polling systems. To explain the formula for $\QL$ let us introduce two additional PGF's. Let
$\Ss_{b_i}(\cdot)$ and $\Ss_{c_i}(\cdot)$ be the PGF's of the joint queue length distribution at service beginnings and service completions, respectively. \citet{Eisenberg72} established a link between the various PGF's: $\V_{b_i}$, $\V_{c_i}$, $\Ss_{b_i}$ and $\Ss_{c_i}$. In fact, from his observation it follows that
it is enough to determine, for example, $\V_{b_i}$, $i=1,\ldots,N$, to know all the remaining PGF's $\V_{c_i}$, $\Ss_{b_i}$, $\Ss_{c_i}$, $i=1,\ldots,N$. Hence, as noted before, all these PGF's are well known for a vast class of polling systems.
One of our main results -- \autoref{thm:main} -- reveals that the formula for $\QL$ can be expressed in terms of $\Ss_{c_i}$, $i=1,\ldots,N$ (or equivalently in any other aforementioned PGF's). Using a similar technique that allowed us to determine $\QL$, we were also able to find $\WO(\cdot)$, the Laplace-Stieltjes transform (LST) of the joint workload distribution in steady state.  We give a formula for $\WO$ in \autoref{thm:main2} expressed in $\V_{b_i}$ and $\V_{c_i}$, which is the second contribution of our paper.

$\QL$ was found in \citet{Sidi92} for polling systems restricted to the case of 
cycles with either all disciplines being gated
or all being exhaustive. $\WO$ was found in \cite{Boxma09}, for the class of polling systems with L\'evy input, but restricted to the case of branching-type service disciplines. Both of the results from \cite{Boxma09,Sidi92} assume non-zero switchover times. It should be emphasized therefore that our main results -- \autoref{thm:main} and \autoref{thm:main2} -- hold regardless of the service discipline and provide a unified expression for $\QL$ and $\WO$ for both the zero and non-zero switchover case; see also \autoref{rem:Sidi} and \autoref{rem:Boxma}.

From \autoref{thm:main2} we were also able to retrieve the well known decomposition of the total workload $W$ in a polling system into an independent sum of the amount of work $W_{M/G/1}$ in a corresponding $M/G/1$ queue and $W_{switch}$, the total amount of work in the polling system at an arbitrary epoch in a switching period. This decomposition was found in \cite{Boxma87}, however the distribution of $W_{switch}$ was only known in a few cases, cf. \cite{Takagi92}. In \autoref{rem:decomp} we present the LST of the distribution of $W_{switch}$ expressed in $\V_{b_i}$ and $\V_{c_i}$. This constitutes the third and final contribution of our paper.

The paper is organised as follows. In \autoref{sec:GPM} we give a detailed description
of the polling model under consideration, introduce the various PGF's $\V_{b_i}, \V_{c_i}, \Ss_{b_i}, \Ss_{c_i}$ and relate them to each other. \autoref{sec:MR} contains our first main result, \autoref{thm:main} -- the expression for $\QL$,
the PGF of the joint queue length distribution in steady state at an arbitrary epoch, in terms of $\Ss_{c_i}$.
\autoref{sec:work} contains our second main result, \autoref{thm:main2} -- the expression for $\WO$, the LST of the joint workload distribution in steady state at an arbitrary epoch, in terms of $\V_{b_i}$ and $\V_{c_i}$ . This section also contains \autoref{rem:decomp}, which establishes the workload decomposition in polling systems.

\section{Model description}
\label{sec:GPM}
\subsection{A general polling system}
\label{sec:GPM:1}
We consider a system of $N\ge 1$ infinite-buffer queues $Q_1,\ldots,Q_N$ and a 
single server $S$.
The service times of customers in $Q_i$ are i.i.d. (independent, identically distributed) positive
random variables generically denoted by $B_i$.
We denote the LST of $B_i$ by $\tilde B_i(\cdot)$ and assume that $b_i:=\ee B_i<\infty$. The server moves among the queues
in a cyclic order. When $S$ moves from $Q_i$ to $Q_{i+1}$, it incurs a switchover period. The durations of successive switchover times are i.i.d. non-negative random variables, which we generically denote by $S_i$. We denote the LST of $S_i$ by $\tilde S_i(\cdot)$ and assume that $s_i:=\ee S_i<\infty$; let $s:=\sum_{i=1}^N s_i$.
Customers arrive at $Q_i$ according to a Poisson process with rate 
$\lambda_i$; let $\lambda:=\sum_{i=1}^N\lambda_i$. We do not assume anything
about the service disciplines at $Q_i$.
Define $\rho_i:=\lambda_i b_i$ as
the traffic intensity at $Q_i$; let $\rho:=\sum_{i=1}^N\rho_i$.
In what follows we shall write $\bz$ for
an $N$-dimensional vector in $\rr^N$, $\bz=(z_1,\ldots,z_N)$, and we assume that $|z_i|<1$ for
every $i=1,\ldots, N$.  Throughout the paper we implicitly use the convention that any index summation is modulo $N$, for example $Q_{N+1}\equiv Q_1$.

We assume that all the usual independence assumptions hold between the service times,
the switchover times and the interarrival times. We assume that the ergodicity conditions are fulfilled 
and we restrict ourselves to results for the stationary situation. 

We also consider the variant of this model in which {\it all} the switchover times are {\it zero}.
In the latter model, $S$ makes a full cycle (viz., passes all the queues
once) when the system becomes empty and subsequently stops right before $Q_1$. All this requires zero time.
When the first new customer
arrives, $S$ cycles along the queues to that customer. Note
that the behavior of the server when the system is empty
does {\it not} affect the queue length distributions.
However, it may involve a number of (zero-length)
visits, and hence it {\it does} affect the queue length distribution
embedded at visit beginning and visit completion instants. The
convention that the server stops at some fixed queue was
used in \cite{Borst97} and turned
out to result in cleaner expressions for that embedded
queue length distribution than when the server is assumed
to stop right away. Making a full cycle before stopping may
be necessary if there is no central queue length information
available to the server so that the only way to detect
an empty system is to keep track if some fixed queue is
visited twice in zero time. 

\subsection{Distributional identities}
\label{sec:GPM:2}
For details of this subsection we refer to \cite{Borst97}.

Note that each time a visit beginning or a service completion occurs, this coincides
with either a service beginning or a visit completion.
All service beginning epochs in a visit to $Q_i$ are also service completion epochs at $Q_i$,
except the first service beginning epoch, which is also a visit beginning epoch.
Similarly, all service completion epochs at $Q_i$ are also service beginning epochs at that queue,
except the last service completion epoch, which is also a visit completion epoch.
This has been observed
by many authors, but most notably by \citet{Eisenberg72},
who restricted himself to the cases
of either exhaustive or gated service at all queues.
However, its applicability is not restricted to a particular service discipline.
After some manipulations this observation yields:
\begin{equation}
\label{eq:equilibrium}
\gamma_i \V_{b_i}(\bz)+\Ss_{c_i}(\bz)= \Ss_{b_i}(\bz)+\gamma_i \V_{c_i}(\bz).
\end{equation}
Here $\V_{b_i}(\bz)$ and $\V_{c_i}(\bz)$ denote the PGF's of the joint queue length distribution at visit beginnings
and visit completions at $Q_i$, while $\Ss_{b_i}(\bz)$ and $\Ss_{c_i}(\bz)$ denote the PGF's of the joint
queue length distribution at service beginnings and service completions at $Q_i$;
the coefficient $\gamma_i$ represents the reciprocal of the mean number of customers served at $Q_i$ 
per visit, i.e., the ratio of visit beginnings to service beginnings.
With $C$ denoting the steady-state cycle length, $1/\gamma_i=\lambda_i\ee C$.
In the case of non-zero switchover times, we simply have $\ee C=s/(1-\rho)$.
In the zero switchover times case
\begin{equation}
\label{eq:zero:3}
\frac{\V_{b_1}(0)}{\ee C}=\lambda(1-\rho).
\end{equation}
Indeed, $\V_{b_1}(0)/\ee C$ denotes the mean number of times per time unit that $S$ arrives
at $Q_1$ to find the whole system empty ($\V_{b_1}(0)$ is the probability that $S$
sees an empty system at a $Q_1$ visit). Each of those epochs is followed by exactly
one arrival of a customer to an empty system. By PASTA, there are on average $\lambda(1-\rho)$ such arrivals per time unit.

We rewrite \eqref{eq:equilibrium} into:
\begin{equation}
\label{eq:equilibriumA}
\gamma_i (\V_{b_i}(\bz) - \V_{c_i}(\bz)) = \Ss_{b_i}(\bz)- \Ss_{c_i}(\bz).
\end{equation}

Let $\Sigma(\bz):=\sum_{j=1}^N \lambda_j(1-z_j)$.
It is easy to see that $\Ss_{c_i}(\bz)$ and $\Ss_{b_i}(\bz)$ are related via
\begin{equation}
\label{eq:relation}
\Ss_{c_i}(\bz)=\Ss_{b_i}(\bz)\frac{\tilde B_i(\Sigma(\bz))}{z_i}.
\end{equation}
It follows from \eqref{eq:equilibrium} and \eqref{eq:relation} that
\begin{equation}
\label{eq:Sbi}
\Ss_{b_i}(\bz)=\frac{\gamma_i z_i}{z_i-\tilde B_i(\Sigma(\bz))}\left(\V_{b_i}(\bz)-\V_{c_i}(\bz)\right).
\end{equation}

Next we relate $\V_{b_{i+1}}$ to $\V_{c_i}$.
In a polling system with switchover times
\begin{equation}
\label{eq:rel}
\V_{b_{i+1}}(\bz)=\V_{c_i}(\bz)\tilde S_i\left(\Sigma(\bz)\right),\quad i=1,\ldots,N.
\end{equation}
In a polling system without switchover times
\begin{equation}
\label{eq:zero:1}
\V_{b_{i+1}}(\bz)=\V_{c_i}(\bz),\quad i=1,\ldots,N-1.
\end{equation}
The relation between $\V_{b_1}$ and $\V_{c_N}$ deserves
special attention in the zero switchover case, because of our convention stated in \autoref{sec:GPM:1},
concerning the behavior of the server when the system is empty.
When all
queues in the model without switchover times become
empty, S in our convention makes a full cycle and subsequently
stops right before $Q_1$ (all this requires no time).
When the first new customer arrives, S cycles along the
queues to that customer. The consequence of this is that
when the system is empty at the start of a visit to $Q_1$, then
the next visit to $Q_1$ does not take place until a customer
has arrived. Therefore,
\begin{equation}
\label{eq:zero:2}
\V_{b_1}(\bz)=\V_{c_N}(\bz)-\V_{b_1}(0)\left(1-\sum_{i=1}^N\frac{\lambda_i}{\lambda}z_i\right)
=\V_{c_N}(\bz)-\frac{\V_{b_1}(0)}{\lambda}\Sigma(\bz).
\end{equation}
In particular, observe that $\V_{b_1}(0)=\V_{C_N}(0)/2$,
due to the specific server movement in an empty system.

\section{Queue length distribution at an arbitrary time}
\label{sec:MR}
\begin{theorem}
\label{thm:main}
For a general polling system as introduced in \autoref{sec:GPM},
let $\QL(\cdot)$ be the probability generating function of the joint queue length distribution at an arbitrary time in steady-state.
Then, 
\begin{equation}
\label{eq:QL}
\QL(\bz)=
\frac{1}{\ee C}\sum_{i=1}^N\left(
\frac{\V_{b_i}(\bz)-\V_{c_i}(\bz)}{\Sigma(\bz)}
\frac{z_i\left(1-\tilde B_i(\Sigma(\bz))\right)}{z_i-\tilde B_i(\Sigma(\bz))}
+
\frac{\V_{c_i}(\bz)-\V_{b_{i+1}}(\bz)}{\Sigma(\bz)}
\right),
\end{equation}
where $\ee C=s/(1-\rho)$ in the non-zero switchover case and $\ee C=\V_{b_1}(0)/(\lambda(1-\rho))$ in
the zero switchover case.
Equivalently,
\begin{equation}
\label{eq:thm:main}
\QL(\bz)=\frac{\sum_{i=1}^N\lambda_i(1-z_i) \Ss_{c_i}(\bz)}{\sum_{i=1}^N
\lambda_i(1-z_i)}.
\end{equation}
\end{theorem}
\begin{proof} We will divide the proof into three parts.
First we prove \eqref{eq:QL} for $\QL$ in the
non-zero switchover times case. Next we will show that this formula is also valid in the zero switchover times case. Finally, we will show that the unified formula can be expressed as \eqref{eq:thm:main}.

Let $X_i(\cdot)$ and $Y_i(\cdot)$
be the PGF of the joint queue length distribution at an arbitrary moment during
a visit to $Q_i$ and during a switchover (idle) period between $Q_i$ and $Q_{i+1}$, respectively.\\
{\it Non-zero switchover times:}
By the stochastic mean value theorem
\begin{equation}
\label{eq:smvt}
\QL(\bz)=\frac{1}{\ee C}\sum_{i=1}^N\left(\frac{b_i}{\gamma_i}
X_i(\bz)
+
s_iY_i(\bz)
\right).
\end{equation}
Furthermore, for any $i$,
\begin{align}
\label{eq:Xi}
X_i(\bz)&=\Ss_{b_i}(\bz) \tilde B_i^{\text{past}}(\Sigma(\bz))
=
\Ss_{b_i}(\bz) \frac{1-\tilde B_i(\Sigma(\bz))}{\Sigma(\bz)b_i}
=
\frac{\gamma_i}{b_i}\frac{z_i\left(\V_{b_i}(\bz)-\V_{c_i}(\bz)\right)}{z_i-\tilde B_i(\Sigma(\bz))}
\frac{1-\tilde B_i(\Sigma(\bz))}{\Sigma(\bz)}
\end{align}
where $\tilde B_i^{\text{past}}(\cdot)$ is the LST of $B_i^{\text{past}}$, the past part
of $B_i$ and where we used \eqref{eq:Sbi}.
Analogously, 
\begin{equation}
\label{eq:Yi}
Y_i(\bz)=\V_{c_i}(\bz) \tilde S_i^{\text{past}}(\Sigma(\bz))
=
\V_{c_i}(\bz) \frac{1-\tilde S_i(\Sigma(\bz))}{\Sigma(\bz)s_i}
=
 \frac{1}{s_i}\frac{\V_{c_i}(\bz)-\V_{b_{i+1}}(\bz)}{\Sigma(\bz)},
\end{equation} 
where $\tilde S_i^{\text{past}}(\cdot)$ is the LST of $S_i^{\text{past}}$, the past part of $S_i$
and where we used \eqref{eq:rel}. 
Combining \eqref{eq:smvt}-\eqref{eq:Yi} leads to \eqref{eq:QL}.
\\
{\it Zero switchover times:}
Observe that in this case we can write
\[
\QL(\bz)=\rho \QL_{serving}(\bz)+(1-\rho) \QL_{non-serving}(\bz), 
\]
where $\QL_{serving}(\cdot)$ and $\QL_{non-serving}(\cdot)$ denote the PGF of the joint queue length distribution at an arbitrary
moment when $S$ is serving and non-serving, respectively. Trivially, $\QL_{non-serving}(\bz)\equiv 1$.

The stochastic mean value theorem gives
\[
\QL_{serving}(\bz)=\frac{1}{\sum_{i=1}^N b_i/\gamma_i}\sum_{i=1}^N \frac{b_i}{\gamma_i}X_i(\bz).
\]
Note that
\[
\frac{\rho}{\sum_{i=1}^N b_i/\gamma_i}=
\frac{\rho}{\ee C\sum_{i=1}^N b_i\lambda_i}=
\frac{\rho}{\ee C\sum_{i=1}^N \rho_i}=\frac{1}{\ee C}.
\]
Hence, in the zero switchover case,
\[
\QL(\bz)=
\frac{1}{\ee C}\sum_{i=1}^N\left(
\frac{\V_{b_i}(\bz)-\V_{c_i}(\bz)}{\Sigma(\bz)}
\frac{z_i\left(1-\tilde B_i(\Sigma(\bz))\right)}{z_i-\tilde B_i(\Sigma(\bz))}
\right)
+
1-\rho.
\]
Noting that \eqref{eq:zero:3}, \eqref{eq:zero:1} and \eqref{eq:zero:2} give
\[
\frac{1}{\ee C}\sum_{i=1}^N\frac{\V_{c_i}(\bz)-\V_{b_{i+1}}(\bz)}{\Sigma(\bz)}
=\frac{\V_{c_N}(\bz)-\V_{b_1}(\bz)}{\ee C\Sigma(\bz)}
=\frac{\V_{b_1}(0)}{\lambda\ee C}=1-\rho,
\]
we conclude that \eqref{eq:QL} is also valid in the zero switchover times case.\\
{\it Simplification of \eqref{eq:QL}:}
Observe that
\[
\sum_{i=1}^N\frac{\V_{c_i}(\bz)-\V_{b_{i+1}}(\bz)}{\Sigma(\bz)}
=
\sum_{i=1}^N\frac{\V_{c_i}(\bz)-\V_{b_i}(\bz)}{\Sigma(\bz)}.
\] 
Hence, using \eqref{eq:relation}, \eqref{eq:Sbi} and the definition of $\gamma_i$ we have
\begin{align*}
\sum_{i=1}^N\Bigg(
\frac{\V_{b_i}(\bz)-\V_{c_i}(\bz)}{\Sigma(\bz)}&\frac{z_i\left(1-\tilde B_i(\Sigma(\bz))\right)}{z_i-\tilde B_i(\Sigma(\bz))}
+
\frac{\V_{c_i}(\bz)-\V_{b_{i+1}}(\bz)}{\Sigma(\bz)}\Bigg)\\
&=
\sum_{i=1}^N\left(
\frac{\V_{b_i}(\bz)-\V_{c_i}(\bz)}{\Sigma(\bz)}\left[
\frac{z_i\left(1-\tilde B_i(\Sigma(\bz))\right)}{z_i-\tilde B_i(\Sigma(\bz))}
-1\right]\right)\\
&=
\sum_{i=1}^N\left(
\frac{\V_{b_i}(\bz)-\V_{c_i}(\bz)}{\Sigma(\bz)}\frac{\tilde B_i(\Sigma(\bz))(1-z_i)}{z_i-\tilde B_i(\Sigma(\bz))}\right)\\
&\stackrel{\eqref{eq:Sbi}}{=}
\sum_{i=1}^N
\frac{1-z_i}{\Sigma(\bz)}\frac{\Ss_{c_i}(\bz)}{\gamma_i}
\stackrel{\eqref{eq:relation}}{=}
\ee C\sum_{i=1}^N
\frac{\lambda_i(1-z_i)}{\Sigma(\bz)}\Ss_{c_i}(\bz),
\end{align*}
which completes the proof.
\end{proof}
We would like to re-emphasise that $\QL$ is completely determined when the $\Ss_{c_i}$ are known.
Thanks to the relations in \autoref{sec:GPM:2}, this is equivalent to knowing, for example, $\V_{b_i}$. The latter functions were found for various polling models as mentioned in \autoref{sec:INT}.
\begin{remark}
\label{rem:Sidi}
In \citet{Sidi92}
the authors also focus on the joint queue length distribution at an arbitrary epoch.
However, the result in \cite{Sidi92} is restricted to the case of non-zero switchover times and
cycles with either all policies being gated
or all being exhaustive.
The authors averaged over $N$ visit times and $N$ switchover times, but did not obtain 
our \autoref{thm:main}. Our theorem reveals that exactly the same structure holds, regardless of the service discipline. 
However, they do allow a more general customer behavior;
their paper is one of the few polling studies in which the system is viewed as a network, with
customers moving from queue to queue and the server visiting the queues cyclically.
\end{remark}
\begin{remark}
Observe that \autoref{thm:main} immediately gives the formula for the marginal distributions.
Indeed, for a vector $\bz_{M,i}=(1,\ldots,1,z_i,1,\ldots,1)$, $\QL(\bz_{M,i})=\Ss_{c_i}(\bz_{M,i})$.
From the `step' (level crossing) argument it follows that $\Ss_{c_i}(\bz_{M,i})$ is also the PGF of
the queue length distribution in $Q_i$ at an {\em arrival} epoch
at $Q_i$. By PASTA it is also the steady-state distribution of $Q_i$.

Next take $\bz_T=(z,\ldots,z)$. \autoref{thm:main} now states that
the PGF of the distribution of the total queue length (in terms of $z$) equals
$\sum_{i=1}^N \lambda_i \Ss_{c_i}(\bz_T)/\sum_{j=1}^N \lambda_j$.
This formula may be interpreted as follows. By PASTA, $\QL(\bz_T)$ is also the PGF
of the distribution of the total queue length at an arrival epoch.
By a level crossing argument, it follows that this equals the PGF of the distribution
of the total queue length just after a departure epoch.
The result now follows from the observation
that a fraction $\lambda_i / \sum_{j=1}^N \lambda_j$
of the departure epochs refers to a departure from $Q_i$.
\end{remark}
\begin{remark}
It would be very interesting to have an interpretation of  \eqref{eq:thm:main}.
Below we present a short derivation, which essentially combines several of the formulas above,
but which fails to give an immediate explanation:
\begin{align*}
\QL(\bz) \stackrel{\eqref{eq:smvt}}{=}&  \sum_{i=1}^N \rho_i X_i(\bz) + (1-\rho) \sum_{i=1}^N \frac{s_i}{s} Y_i(\bz)
\\
\stackrel{\eqref{eq:Xi},\eqref{eq:Yi}}{=} &
\sum_{i=1}^N \rho_i \frac{\Ss_{b_i}(\bz) - z_i \Ss_{c_i}(\bz)}{b_i \sum_{j=1}^N \lambda_j(1-z_j)}
+ \frac{1}{\ee C} \sum_{i=1}^N \frac{\V_{c_i}(\bz) - \V_{b_{i+1}}(\bz)}{\sum_{j=1}^N \lambda_j(1-z_j)}
\\
= &
\frac{\sum_{i=1}^N \lambda_i (\Ss_{b_i}(\bz) - z_i \Ss_{c_i}(\bz))}{\sum_{j=1}^N \lambda_j(1-z_j)}
+ \frac{1}{\ee C} \frac{\sum_{i=1}^N (\V_{c_i}(\bz) - \V_{b_i}(\bz))}{\sum_{j=1}^N \lambda_j(1-z_j)}
\\
\stackrel{\eqref{eq:equilibriumA}}{=} &
\frac{\sum_{i=1}^N \lambda_i (\Ss_{b_i}(\bz) - z_i \Ss_{c_i}(\bz))}{\sum_{j=1}^N \lambda_j(1-z_j)}
+ \frac{\sum_{i=1}^N \lambda_i (\Ss_{c_i}(\bz) - \Ss_{b_i}(\bz))}{\sum_{j=1}^N \lambda_j(1-z_j)}
\\
=& 
\frac{\sum_{i=1}^N\lambda_i(1-z_i) \Ss_{c_i}(\bz)}{\sum_{i=1}^N
\lambda_i(1-z_i)}.
\end{align*}
\end{remark}

\section{Workload distribution at an arbitrary time}
\label{sec:work}
\newcommand{\Bo}{\tilde{ \boldsymbol B}(\bo)}
The idea of the proof of \autoref{thm:main} can be also used to determine the LST $\WO$ of the joint workload distribution
at an arbitrary epoch.
For future use, let  $\bo := (\omega_1,\dots,\omega_N)$ and
$\Bo:=(\tilde B_1(\omega_1),\ldots,\tilde B_N(\omega_N))$.
Moreover, recall that $\Sigma(\bz)=\sum_{j=1}^N\lambda_j(1-z_j)$.
\begin{theorem}
\label{thm:main2}
For a general polling system as introduced in \autoref{sec:GPM}, 
let $\WO(\cdot)$ be the Laplace-Stieltjes transform of the joint workload distribution at an arbitrary time in steady-state.
Then,
\[
\WO(\bo) =
\frac{1}{\ee C}\sum_{i=1}^N
\frac{\V_{b_i}(\Bo)-\V_{c_i}(\Bo)}{\Sigma(\Bo)}
\frac{\omega_i}{\Sigma(\Bo)-\omega_i},
\]
where $\ee C=s/(1-\rho)$ in the non-zero switchover case and $\ee C=\V_{b_1}(0)/(\lambda(1-\rho))$ in
the zero switchover case.
\end{theorem}
\begin{proof} The proof follows the same reasoning as the proof
of \autoref{thm:main}. In particular, we focus only on the non-zero switchover times case.

Let $\tilde{X}_i(\cdot)$ and $\tilde{Y}_i(\cdot)$
be the LST of the joint workload distribution at an arbitrary moment during
a visit to $Q_i$ and during a switchover (idle) period between $Q_i$ and $Q_{i+1}$, respectively. Then
by the stochastic mean value theorem
\[
\WO(\bo)=\frac{1}{\ee C}\sum_{i=1}^N\left(\frac{b_i}{\gamma_i}
\tilde{X}_i(\bo)
+
s_i \tilde{Y}_i(\bo)
\right).
\]
Firstly, note that using \eqref{eq:Yi},
\[
\tilde Y_i(\bo)= Y_i(\Bo)
=
\frac{1}{s_i}\frac{\V_{c_i}(\Bo)-\V_{b_{i+1}}(\Bo)}{\Sigma(\Bo)}.
\]
Secondly,
\[
 \hat{X}_i(\bo) =\frac{\Ss_{b_i}(\Bo)}{\tilde{B}_i(\omega_i)} 
\times \int_{u=0}^{\infty} \int_{t=0}^{\infty}
\exp\left(- \Sigma(\Bo) u
- \omega_i t \right)
{\rm d}\prob{B_i^{past} < u, B_i^{res} < t}.
\]
Indeed, consider an arbitrary service time during a visit to $Q_i$. The term 
$\Ss_{b_i}(\Bo)/\tilde{B}_i(\omega_i)$
corresponds to the LST of the workload of all the customers in the system at the beginning of that service time,
{\em excluding} the customer whose service is about to begin.
Secondly, consider the workload that arrives at all queues during the
past part
$B_i^{past}$ of that service time, {\em and} the residual part $B_i^{res}$ of that same service time.
Next integrate with respect to the joint distribution of $B_i^{past}$ and $B_i^{res}$.
Using \eqref{eq:Sbi} and the fact that, cf. \cite[Section I.6.3]{Cohen82},
\[
\int_{u=0}^{\infty} \int_{t=0}^{\infty} \exp\left(- \Sigma(\Bo) u
- \omega_i t \right){\rm d}\prob{B_i^{past} < u, B_i^{res} < t} = \frac{\tilde{B}_i(w_i) - \tilde{B}_i(\Sigma(\Bo))}
{b_i\left(\Sigma(\Bo)-w_i\right)},
\]
we obtain
\begin{align*}
\hat{X}_i(\bo) 
&= 
\frac{\Ss_{b_i}(\Bo)}{\tilde{B}_i(\omega_i)} 
\frac{\tilde{B}_i(w_i) - \tilde{B}_i(\Sigma(\Bo))}
{b_i\left(\Sigma(\Bo)-w_i\right)}
=
\frac{\gamma_i}{b_i}
\frac{\V_{b_i}(\Bo)-\V_{c_i}(\Bo)}
{\tilde{B}_i(\omega_i) - \tilde{B}_i(\Sigma(\Bo))}
\frac{\tilde{B}_i(\omega_i) - \tilde{B}_i(\Sigma(\Bo))}
{\Sigma(\Bo) - \omega_i}
\\
&= 
\frac{\gamma_i}{b_i}
\frac
{\V_{b_i}(\Bo)
-
\V_{c_i}(\Bo)}
{\Sigma(\Bo) - \omega_i} .
\end{align*}
This completes the proof.
\end{proof}
\begin{remark}
\label{rem:Boxma}
$\WO$ was found in \cite{Boxma09} for the more general class of polling systems with L\'evy input,
but restricted to the case of branching-type service disciplines at all queues and non-zero switchover times.
Notice that $\Sigma(\Bo)-\omega_i$ is the Laplace exponent $\phi_i^{\boldsymbol A}(\bo)$ of the L\'evy process \[\boldsymbol A_i(t)=(W_1(t),\ldots,W_{i-1}(t),W_i(t)-t,W_{i+1}(t),\ldots,W_{N}(t)),\]
where $W_i$ is a compound Poisson process with jump distribution $B_i$ and rate $\lambda_i$. Moreover $\Sigma(\Bo)$ is the Laplace exponent of the L\'evy process $\boldsymbol W(t)=(W_1(t),\ldots,W_N(t))$. 
After identifying
$\V_{b_i}(\Bo)$ and $\V_{c_i}(\Bo)$ of \autoref{thm:main2} with $\tilde{\boldsymbol B}_i(\bo)$ and $\tilde{\boldsymbol E}_i(\bo)$ of Theorem 3 of \citep{Boxma09},
it can now be verified that \autoref{thm:main2} indeed
coincides with \cite[Theorem 3]{Boxma09}.
\end{remark}
\begin{remark}
\label{rem:decomp}
Observe that \autoref{thm:main2} gives the LST of the distribution
of $W$, the total workload in a polling system. With $\bo_T=(\omega,\ldots,\omega)$, we have $\ee[{\rm e}^{-\omega W}]=\WO(\bo_T)$, so that
\[
\ee[{\rm e}^{-\omega W}] =
\frac{1-\rho}{s} 
\frac{\omega}{\omega-\sum_{j=1}^N \lambda_j(1 - \tilde{B}_j(\omega))}
\sum_{i=1}^N 
\frac{\V_{c_i}(\tilde{\boldsymbol B}(\bo_T)) -
\V_{b_i}(\tilde{\boldsymbol B}(\bo_T))}
{\sum_{j=1}^N \lambda_j(1 - \tilde{B}_j(\omega))}.
\]
Note that the LST of the amount of work $W_{M/G/1}$ in a corresponding $M/G/1$ queue,
that is, an $M/G/1$ queue with arrival rate $\lambda$ 
and service time LST $\sum_{j=1}^N \lambda_j \tilde{B}_j(\omega) / \lambda$, is given by
\[
\ee[{\rm e}^{-\omega W_{M/G/1}}] = \frac{(1-\rho) \omega}{\omega - \sum_{j=1}^N \lambda_j (1 - \tilde{B}_j(\omega))}.
\]
Note that $W_{M/G/1}$ is also
the amount of work in the polling system with zero switchover times.
Finally, from the proof of \autoref{thm:main2} we know that the LST
of the distribution of the total amount of work $W_{switch}$ in a polling system at an arbitrary epoch in a switching period is given by
\begin{equation}
\label{eq:Wswitch}
\ee[{\rm e}^{-\omega W_{switch}}] =
\frac{1}{s} 
\sum_{i=1}^N s_i\tilde Y_i(\bo_T)
=\frac{1}{s}
\sum_{i=1}^N
\frac{\V_{c_i}(\tilde{\boldsymbol B}(\bo_T))
-
\V_{b_i}(\tilde{\boldsymbol B}(\bo_T))}
{\sum_{j=1}^N \lambda_j (1-\tilde{B}_j(\omega))}.
\end{equation}
Hence, we retrieve the well known decomposition (see, e.g., \cite{Boxma87})
\[
\ee[{\rm e}^{-\omega W}] =
\ee[{\rm e}^{-\omega W_{M/G/1}}]
\ee[{\rm e}^{-\omega W_{switch}}].
\]
The distribution of $W_{switch}$ was only known in a few
special cases, cf. \cite{Takagi92}. From \autoref{thm:main2} it follows that in general its LST is given by \eqref{eq:Wswitch}.
\end{remark}

\section{Acknowledgements}
The authors would like to thank Marko Boon for reading the first draft of this paper, which led to the nice
simplification of Formula \eqref{eq:QL} into \eqref{eq:thm:main}.

\small
\bibliography{Polling}
\end{document}